\documentclass[11pt, reqno]{amsart}
\usepackage{a4wide,amsfonts,amsmath,amssymb,amsthm,epsfig,cite,graphicx,hyperref,
color, esint,fancyhdr, enumerate, latexsym, amsrefs}

\numberwithin{equation}{section}

\usepackage{tikz-cd}

\newtheorem*{theorem*}{Theorem}
\newtheorem*{corollary*}{Corollary}
\newtheorem{thm}{Theorem}[section]
\newtheorem{prop}[thm]{Proposition}
\newtheorem{lem}[thm]{Lemma}
\newtheorem{cor}[thm]{Corollary}
\theoremstyle{definition}

\newcommand{\pa}{\partial}

\newcommand{\Om}{\Omega}

\renewcommand{\Re}{\operatorname{Re}}

\hypersetup{
    colorlinks,
    citecolor=blue,
    filecolor=black,
    linkcolor=red,
    urlcolor=black
}

\newcommand{\h}{\mathcal{H}}
\newcommand{\D}{\mathbb{D}}
\newcommand{\C}{\mathbb{C}} 
\newcommand{\R}{\mathbb{R}}

\newcommand{\N}{\mathbb{N}}

\begin{document}
\title{Boundary behaviour of the Span metric and its higher-order curvatures}

\author{Amar Deep Sarkar}

\address{ADS: Harish-Chandra Research Institute, Allahabad (Prayagraj) 211019, India}
\email{amardeepsarkar@hri.res.in}

\keywords{Span metric, higher-order curvatures, reduced Bergman kernel, Dirichlet space, reproducing kernel}

\subjclass{32A07, 30F45, 32A25}	

\thanks{The author was supported by the postdoctoral fellowship of Harish-Chandra Research Institute, Allahabad.}	
	
%
%
%
%
%

\begin{abstract}
In this note, we use scaling principle to study the boundary behaviour of the span metric and its higher-order curvatures on finitely connected Jordan planar domains. A localization of this metric near boundary points of finitely connected Jordan domains is also obtained. Further, we obtain boundary sharp estimates for this metric on $ C^2 $-smooth bounded domains and consequently, this metric is comparable to the Carath\'eodory and the Kobayashi metrics on these domains.
\end{abstract}

\maketitle

\section{Introduction}
\noindent
Let $ D $ be a finitely connected Jordan domain in the complex plane $ \C $ -- by this, we mean that the boundary of $ D $ is bounded by finitely many Jordan curves. The goal of this article is to use the scaling principle to understand the boundary behaviour
of the span metric and its higher-order curvatures on the planar domain $ D $.
\medskip 

Now, we shall define the span metric.
To define the span metric, first, let us define the Dirichlet space of the domain $ D $ at a fixed point $ \zeta \in D $. Let $ \mathcal{O}(D) $ be the set of all holomorphic functions on $ D $.

\medskip

\medskip

The Dirichlet space of the domain $ D $ at a fixed point $ \zeta \in D $ is defined by
\[
\mathcal{D}_{\zeta}(D):= \left\{f \in \mathcal{O}(D): f(\zeta) = 0 \,\, \text{and} \,\,  \int_{D}|f^{\prime}(z)|^2dA(z) < \infty \right\}.
\]
This space is a Hilbert space with respect to the inner product
\[
\mathcal{D}(f, g) := \int_{D}f^{\prime}(z) \overline{g^{\prime}(z)}dA(z). 
\]
The evaluation functional 
\[
f \to f^{\prime}(\zeta)
\]
is a bounded linear functional on $ \mathcal{D}_{\zeta}(D) $.
Using the reproducing kernel Hilbert space property of the Dirichlet space, we have
\[
f^{\prime}(\zeta) = \int_{D}f^{\prime}(z) \overline{M^{\prime}(z, \zeta)}dA(z), 
\]
where $ M(z, \zeta) $ is the reproducing kernel, and its derivative, i.e., $ \tilde{K}_D(z, \zeta) := M^{\prime}(z, \zeta) $, is called the reduced Bergman kernel. Here we use the subscript to emphasize dependence on the domain  $ D $.
\medskip

The kernel $ \tilde{K}_D(z, \zeta) $ is holomorphic in the first variable and anti-holomorphic in the second variable, and 
\[ 
\tilde{K}_D(z, \zeta) = \int_D \tilde{K}_D(\xi, \zeta) \overline{ \tilde{K}_D(\xi, z)}dA(\xi).
\]

For a bounded domain $ D \subset \C $, $ \tilde{K}_D(\zeta, \zeta) > 0 $.
\medskip

Now, consider the extremal problem,
\[
s_D(\zeta) := \sup\left\{ | f^{\prime}(\zeta) |^2 :  f \in \mathcal{D}(D) \,\, \text{and}\,\, \mathcal{D}(f)\leq \pi \right\},
\]
where $\mathcal{D}(f) := \int_D |f^{\prime}(z)|^2dA(z) $. 
\medskip

The metric $ s_D(z) |dz|^2 $ is called the {\em span metric}. It is a conformal metric, i.e.,
if $ \phi: D_1 \longrightarrow D_2 $ a biholomorphic map between the planar domains $ D_1 $ and $ D_2 $, then
\[
\phi^{*}(s_{D_2})(z) := s_{D_2}(\phi(z))|\phi^{\prime}(z)|^2 = s_{D_1}(z)
\] 
for all $ z \in D_1 $. It is known that, (see \cite[Example 6]{UnifiedMetricSugawa}), that the span metric and the reduced Bergman kernel are related by the following formula:
\[
s_D(z) = \pi \tilde{K}_D(z, z),
\] 
and this guarantees that the span metric is real analytic -- since reduced Bergman kernel is real analytic along the diagonal. 
\medskip

In \cite{AhlforsAndBeurling}, Ahlfors and Beurling study span metrics of domains in the complex plane as an extremal problem of a collection of holomorphic functions with finite Dirichlet integrals, and deduced a few results regarding this metric. They also compare this metric with the Carath\'eodory metric and show that the span metric is smaller than the Carath\'eodory metric. This result, in general, is true for Riemann surfaces, and it was shown by Sakai in \cite{SakaiSpanMetricCurvatureBound}. This metric is also used in the classification problem of Riemann surfaces, see \cite{MakatoActaAnlyticDirichlet}, \cite{SarioClassificationofRiemannSurfaces}. A potential theoretic perspective has also been observed for this metric by Schiffer and Garabedian in \cite[Section 2]{PotentialPerspectiveShiffereGarabedian}.
\medskip

When $ D = \D $, the unit disk, the span metric is equal to the Poincar\'e metric, i.e., $ s_{\D}(z) = 1/(1 - |z|^2)^2 $ for all $ z \in \D $. Therefore, the Gaussian curvature, denoted by $ \kappa(z: s_D) $, of the span metric $ s_D(z)|dz|^2 $ of any simply connected domain $ D $ is equal to $ -4 $. Zarankiewicz in \cite{Zarankiewicz} shows that $ \kappa(z: s_D) < - 4 $ when $ D $ is an annulus  (also see \cite{SpanMetricTriplyConnectedDomains} for triply-connected domains). In \cite{StefanBergmanSpanMetricCurvatureBoundary}, Bergman proves the following: for a boundary point $ p \in \partial D $, $ D \subset \C $ a domain, such that there exist two circles touching at $ p $, one in the interior of $ D $ and the other at the exterior of $ D $,  the Gaussian curvature $ \kappa(z: s_D) $ of the span metric converges to $ -4 $ as $ z \to p $ so that $ z - p $ and the interior normal at $ p $ makes an angle in absolute value less than $ \pi/2 $. Further, Burbea shows that the Gaussian curvature of the span metric is bounded above by $ -2 $, see \cite{BurbeaCurvatureLessThanMinusTwo}. \medskip

In \cite{SuitaConjecture}, Suita conjectured that $ \kappa(z: s_{\Omega}) \leq - 4 $ for every
Riemann surface $ \Omega $ exhibiting holomorphic functions with finite Dirichlet integrals. This conjecture was settled by Sakai, see \cite{SakaiSpanMetricCurvatureBound}, which is an application of sub-mean value property of subharmonic functions. 
\medskip

In \cite{BurbeaSpanMetricHigherOrderCurvature}, Burbea defines the higher-order curvatures of the span metric -- the Gaussian curvature comes as a special case of this -- and gives upper bounds on the higher-order curvatures of the span metric. 
\medskip

To state Burbea's result, first, we shall define the higher-order curvatures, in particular $ n $-th order curvature of the span metric for every $ n \in \N $. It is defined as follows:
\[
\kappa_n(z) = \kappa_n(z: s_D) := -(n + 1)! s_D^{-(n + 1)(n + 2)/2}\det[s_{j \bar k}]_{j, k = 0}^{n},
\]
where
\[
s_{j \bar k}(z) = \frac{\partial^{j + k}s_D}{\partial z^j \partial\overline{z}^{k}}(z).
\]	
Note that $ k_1 = -\Delta \log s_D/2s_D  $ is the standard Gaussian curvature. Also note that the $ n $-th order curvature of the span metric is a conformal invariant.
\medskip

\noindent
{\bf Result}~(Burbea, \cite{BurbeaSpanMetricHigherOrderCurvature}){\bf.}
{\em Let $ D \subset \C $ be a domain such that $ s_D> 0 $ on $ D $ . Then the $ n $-th order curvatures of the span metric $ s_D(z)|dz|^2 $ satisfy
	\[
	\kappa_n(z) \leq -
	\left(\prod_{k = 1}^{n + 1}k!\right)^2,
	\]
	and equality holds if and only if $ D $ is a simply connected domain minus a null set with respect to the metric $ s_D(z)|dz|^2 $.}
\medskip

We have the following results:
\begin{thm}\label{T:MainThm1}
	Let $ D \subset \C $ be a finitely connected bounded domain bounded by Jordan curves, and $ p \in \partial D $ a $ C^2 $-smooth boundary point. Then
	\[
	s_D(z)  \approx 1/{\rm dist} (z, \partial D)
	\]
	for $z$ close to $p$.
\end{thm}
Here and in what follows, we use the standard convention that $ A \approx B $ means that both $ A/B,\,\, B/A $ are bounded above by some constant $ C > 1 $. The theorem above shows that the span metric is uniformly comparable to the quasi-hyperbolic metric near a $ C^2 $-smooth boundary point, consequently, when $ D $ smoothly bounded, the span metric is uniformly comparable to the quasi-hyperbolic metric everywhere on $ D $ (and hence to the Kobayashi and the Carath\'eodory metrics). In fact, a stronger version of the theorem above will be proved: we show that the ratio of the span metric and the quasi-hyperbolic metric converges to a strictly non-zero quantity near a $ C^2 $-smooth boundary point.\medskip


We strengthen the result by Burbea as follows.
\begin{thm}\label{T:MainThm2}
	Let $ D \subset \C $ be a finitely connected bounded domain bounded by Jordan curves, and $ p \in \partial D $ a $ C^2 $-smooth boundary point. Then
	\[
	\kappa_n(z) \to  -\left(\prod_{k = 1}^{n + 1}k!\right)^2
	\]
	as $ z \to p $.
\end{thm}

\medskip

\section{Scaling of planar domains}\label{S:Scaling}
\noindent
The scaling principle is a useful tool to study problems in several complex variables. For a general philosophy and an extensive exposition of the scaling principle, and various problems solved using it on several complex variables, we refer the reader to \cite{ScalingInHigherDimensionKrantzKimGreen}. It also can be used to study the boundary behaviour of conformal metrics and invariants in one complex variable, see \cite{DigantaPravavKushal}, which studies the boundary behaviour of the Green's functions and the capacity metrics on planar domains. Recently, in \cite{BoundaryBehConfAmarVerma}, the scaling method has been used to study the boundary behaviour of the Carath\'eodory metric, Aumann-Caratheodory rigidity constant, and a conformal metric arising from integrable holomorphic quadratic differentials (\cite{SugawaMetric}). A simplified version, which will suffice here, can be described as follows: Let $D \subset \mathbb C$ be a domain and $p \in \partial D$ a $C^2$-smooth boundary point. Let $U$ be a neighbourhood of $ p $ and $ \psi \in C^2(U)$ be such that $U \cap D = \{\psi < 0\}, U \cap \partial D = \{\psi = 0\}$ 
with $d \psi \neq 0$ on $U \cap \partial D$. Let $ p_n $ be a sequence in $ D $ converging to $p$. The affine maps
\begin{equation}\label{Eq:ScalingMap}
T_n(z) = \frac{z - p_n}{-\psi(p_n)}
\end{equation}
satisfy $ T_n(p_n) = 0 $ for all $n \ge 1$ and since $ \psi(p_n) \to 0$, it follows that the $T_n$'s expand a fixed neighbourhood of $ p $, i.e., for a compact set $ K $ in $ \C $, $ K \subset T_n(U) $ for all large $n$. The defining functions of the domains $ D_n' = T_n(U \cap D) $ are given by 
\[
\psi_n(z) = \frac{1}{-\psi(p_n)}\psi \circ T_n^{-1}(z)
\]
and it can be seen that the $\psi_n$'s converge to 
\[
\psi_{\infty}(z) = -1 + 2 \Re\left( \frac{\partial \psi}{\partial z}(p)z \right) 
\]
uniformly on compact subsets of $\C$ as $ n \to \infty $. At this stage, let us recall the Hausdorff metric on subsets of a metric space.
\medskip 

Given a set $ S \subseteq \C$, let $ S_{\epsilon} $ denote the $ \epsilon $-neighbourhood of $ S $ with respect to the standard Euclidean distance on $ \C $. The Hausdorff distance between compact sets $ X, Y \subset \C $ is given by 
\[
d_H(X,Y) = \inf\{\epsilon > 0 : X \subset Y_{\epsilon}\,\, \text{and} \,\, Y \subset X_{\epsilon}\}. 
\]
It is known that this is a complete metric space on the space of compact subsets of $ \C $. To deal with non-compact but closed sets there is a topology arising from a family of local Hausdorff semi-norms. It is defined as follows: fix an $ R > 0 $ and for a closed set $ A \subset \C $, let $ A^R = A \cap \overline B(0, R)$ where $ B(0, R) $ is the ball centred at the origin with radius $ R $. Then, for $ A, B \subset \C$, set
\[
d_H^{(R)}(A, B) = d_H\left(A^R, B^R\right).
\]
We say that a sequence of closed sets $ A_n $ converges to a closed set $ A $ if there exists $ R_0 $ such that 
\begin{equation}\label{E:HaudorffCloseConvergence}
\lim_{n \to \infty}d_H\left(A_n^R, A^R\right) = 0
\end{equation}
for all $ R \geq R_0 $. Furthermore, we say a sequence of open sets $ \Omega_j \subset \C $ converges to an open set $ \Omega \subset \C $ in the Hausdorff sense if the sequence of their closures $ \overline{\Omega}_j $ converges to $ \overline{\Omega} $ and the sequence of their boundaries $\pa \Om_j$ converges to the boundary $\pa \Om$ as in (\ref{E:HaudorffCloseConvergence}). More details on the Hausdorff metric on compact sets can be found, for example, in \cite{FalconerHausdorffMetric}, while the variant for open bounded sets is discussed in \cite{BoasHausdorff}. The scaling method produces a limit domain that is always unbounded, and to take care of this, it is useful to see what happens in fixed but arbitrarily large closed balls.

\medskip 

Since $ \psi_j \to \psi_{\infty} $ uniformly on every compact subset in $ \C $, it follows that the sequence $ D_j' = T_j(U \cap D)$  converges to  the half-space 
\[
\h = \{z : -1 +  2 \Re \left( \frac{\partial \psi}{\partial z}(p)z \right) < 0 \} = \{z: \Re (\overline{\omega}z-1)<0\},
\]
where $\omega = (\partial\psi/\partial x)(p) + i(\partial\psi/\partial y)(p)$, in the Hausdorff sense as described above. As a consequence, every compact $ K \subset \h $ is eventually contained in $ D_j' $ for all $j$  large. Similarly, every compact $K$ that does not intersect $\overline{\h} $ eventually has no intersection with $ \overline{D}_j' $.  
\medskip

It can be seen that the same property holds for the domains $ D_j = T_j(D) $, i.e., they converge to the half-space $ \h $ in the Hausdorff sense. We also note that, as a consequence of Hausdorff convergence, the sequence of scaled domains $ D_j $ converges to the half-space $ \h $ in the Carath\'eodory kernel sense.
\medskip

\section{Localization of the span metric}

Let $ D $ be an $ l $-connected domain bounded by Jordan curves, $ l \in \N $, and $ p \in \partial D $ a $ C^2 $-smooth boundary point. For localization purpose, we may assume, without loss of generality,  that the point $ p $ is in the outer boundary curve which is also the boundary of the unbounded component of $ D^c $. Otherwise, we can transform the domain in the following way and work with the transformed domain: 
assume that $ p $ lies in an inner boundary curve. Let $ a $ be a point in the domain bounded by the curve which contains $ p $, which is a connected component of $ D^c $. The inversion $ T(z) := 1\big/ (z- a) $ transform the domain $ D $  into $ T(D) $, and the boundary point $ T(p) $ is in the outer boundary curve of $ T(D) $. 
\medskip

Since $ D $ is an  $ l $-connected domain, by filling the interior holes of this domain, we can make it into a simply connected domain, i.e., if we take $ D' := \cup_{q = 1}^{l - 1}K_q \cup D $ --- where $ K_q $ are the bounded components of $ D^c $ --- then $ D' $ is a simply connected domain bounded by a Jordan curve. Note that $ p $ lies in the boundary of $ D' $.
\medskip

By the Riemann mapping theorem, there exists a biholomorphic map $ f : D' \longrightarrow \D $, where $ \D $ is the unit disk centred at $ 0 $, and by the Carath\'eodory theorem, this map can be extended such that the extended map is a homeomorphism from  $ \overline {D'} $ onto  to $ \overline \D $. If necessary by rotating, we may assume that $ f(p) = 1 $.
\medskip

By using the fact that the map
\[
g(z) := i\left(\frac{1 - z}{1 + z}\right)
\]
gives a biholomorphism between the unit disk $ \D $ and the upper half plane $  \mathbb{H} $ sending the boundary point $ 1 $ to $ 0 $, we obtain a biholomorphism between $ D' $ onto $ \mathbb{H} $ by composing $ g $ and $ f $, and the extended composed map sends the boundary $ p $ to the boundary point $ 0 $ of the upper half plane. 
\medskip

If we restrict the map $ g \circ f $ to $ D $, it gives a biholomorphism of $ D $ onto $ \mathbb{H} \setminus L $, where $ L = (g \circ f) (\cup_{q = 1}^{n - 1}K_q) $. Thus, we have a biholomorphic map from $ D $ onto $ \mathbb{H} \setminus L $ which take the $ p $ to $ 0 $ and the outer boundary maps to $ \R \cup \{\infty\} $.\medskip

Thus, we have the following lemma:

\begin{lem}\label{L:LocalizationHomeomorphism}
	Let $ D $ be an $ l $-connected domain bounded by Jordan curves. Then there exist an $ l $-connected domain $ H \subset \mathbb{H} $ with outer boundary $ \R $, and a homeomorphism  $ f: \overline{D} \longrightarrow  \overline{H} \cup \{\infty \} $ such that $ f|_D: D \longrightarrow H $ a biholomorphism and $ f(p) = 0$.
\end{lem}

\begin{prop}\label{T:LocalizationOfTheSpanMetric}
	Let $ D $ be an $ l $-connected domain bounded by Jordan curves and $ p \in \partial D $ a  boundary point. Then there exists a neighbourhood $ U $ of $ p $ such that $ U \cap D $ is a simply connected and 
	\[
	\lim_{z \to p}\frac{s_{U \cap D}(z)}{s_D(z)} = 1.
	\] 
\end{prop}

To prove this proposition, we shall need the following lemma:
\begin{lem}\label{L:LocalizationForUpperHalfPlane}
	Let $ D \subset \C $ be an $ l $-connected domain which is subset of the upper half plane $ \mathbb{H} $ and the outer boundary $ \R $. Then there exists $ r>0 $ such that $ U = \{ z \in \C: \left|z \right| < r\} \cap D $ is simply connected and
	\[
	\lim_{z \to 0}\frac{s_{U \cap D}(z)}{s_D(z)} = 1
	\] 
\end{lem}
\begin{proof}
	Since $ D $ is an  $ l $-connected domain with outer boundary $ \R $, there exists $ r > 0 $ such that $ U \cap \mathbb{H} = U \cap D $ is simply connected, where $ U :=  \{ z \in \mathbb{C}: \left|z \right| < r\} $. Note that $ U \cap D = \{z \in \mathbb{H}: |z| < r \} $. Let $ \lambda_D $ be the hyperbolic metric of $ D $.
	\medskip 
	
	Using monotonicity of the hyperbolic metric and the span metric along with the fact that for a simply connected domain the span metric coincides with the hyperbolic metric, we have 
\begin{equation}\label{Eq:MonotonicitySpanHyerbolic}
	\lambda_{\mathbb{H}}(z) = s_{\mathbb{H}}(z) \leq s_{D}(z) \leq s_{U \cap D}(z) = \lambda_{U \cap D}(z)  
\end{equation}
	
	Now, we claim that
	\[
	\lim_{z \to 0}\frac{\lambda_{\mathbb{H}}(z)}{\lambda_{{U \cap D}}(z)} = 1.
	\]
	The hyperbolic metric of $ U \cap D $ is given by 
	\[
	\lambda_{{U \cap D}}(z) = \frac{|r + z||r - z|}{{\rm Im}(z)(r^2 - |z|^2)},
	\]
	and the hyperbolic metric of the half space is given by
	\[
	\lambda_{\mathbb{H}}(z) = \frac{1}{{\text{Im}}(z)},
	\]
	and by taking their ratio, we obtain
	
	\[
	\frac{\lambda_{\mathbb{H}}(z)}{\lambda_{U \cap D}(z)} = \frac{r^2 - |z|^2}{|r + z||r - z|}.
	\]
	Note that the right hand side converges to $ 1 $ as $ z \to 0 $.
	Now, using (\ref{Eq:MonotonicitySpanHyerbolic}), we have
	\[
	\frac{\lambda_{\mathbb{H}}(z)}{\lambda_{{U \cap D}}(z)} \leq \frac{s_{D}(z)}{s_{{U \cap D}}(z)} \leq 1.
	\]
	Hence
	\[
	\lim_{z \to 0}\frac{s_D(z)}{s_{U \cap D}(z)} = 1,
	\] 
	and this completes the proof.
\end{proof}

\begin{proof}[Proof of Proposition \ref{T:LocalizationOfTheSpanMetric}]
	By Lemma~\ref{L:LocalizationHomeomorphism}, there exists a homeomorphism $ f : \overline{D} \longrightarrow \overline{H} \cup \{\infty \} $ such that $ f|_D $ is a biholomorphism onto $ H $ and $ f(p) = 0 $ , where $ H $ is an $ l $-connected domain in $ \mathbb{H} $ with outer boundary $ \R $. Let $ U $ be a neighbourhood of $ p $ such that $ U \cap D $ is a simply connected (this can be done because $ D $ is an $ l $-connected domain). Further, using the continuity of $ f $ and $ f^{-1} $, we can choose the neighbourhood $ U $ such that there exist $ r_1 > r_2 > 0 $ satisfying
	\[
	U_{r_2} := \{w \in \mathbb{H}: \left| w \right| < r_2\} \subset f(U \cap D) \subset \{w \in \mathbb{H}: \left| w \right| < r_1\} =: U_{r_1} \subset f(D).  
	\] 
	Using monotonicity of the span metric
	\[
	\frac{s_{U_{r_1}}(w)}{s_{f(D)}(w)} \leq \frac{s_{f(U \cap D)}(w)}{s_{{f(D)}}(w)} \leq \frac{s_{U_{r_2}}(w)}{s_{f(D)}(w)}
	\]
	for all $ w \in U_{r_2} $.
	
	From above, using Lemma~\ref{L:LocalizationForUpperHalfPlane}, it follows that 
	\begin{equation}\label{Eq:FDratio}
	\lim_{w \to 0}\frac{s_{f(U \cap D)}(w)}{s_{{f(D)}}(w)} = 1.
	\end{equation}
	
	Now, note that the map $ f $ when restricted to $ U \cap D $ is a biholomorphism, and using the pull-back formula, we obtain
	\[
	\frac{s_{U \cap D}(z)}{s_{{D}}(z)} = \frac{s_{f(U \cap D)}(f(z))\left|f^{\prime}(z)\right|}{s_{{f(D)}}(f(z))\left|f^{\prime}(z)\right|} = \frac{s_{f(U \cap D)}(f(z))}{s_{{f(D)}}(f(z))}.
	\]
	As $ z \to p $, $ f(z) \to 0 $, therefore using (\ref{Eq:FDratio}), we obtain
	\[
	\lim_{z \to p}\frac{s_{U \cap D}(z)}{s_{{D}}(z)} = 1.
	\]  
\end{proof}
In what follows, $ D $ denotes an $ l $-connected domain bounded by Jordan curves and $ p \in \partial D $ a $ C^2 $-smooth boundary point, $ U $ a neighbourhood of $ p $ such that $ U \cap D $ a simply connected, $ D_j = T_j(D) $ the scaled domains and $ D_j' = T_j(U \cap D)$, where
\[
T_j(z) = \frac{z - p_j}{-\psi(p_j)}
\]
and $ \psi $ is a local defining function of $ \partial D $ near $ p $.
\begin{thm}\label{T:ConvergenceOfTheReducedBergmanKernel}
	Let $ \tilde{K}_j(z, \zeta) $ be the sequence of reduced Bergman kernels of the scaled domains $ D_j $ and $ \tilde{K}_{\mathcal{H}}(z, \zeta) $ be the reduced Bergman kernel of the half space $ \mathcal{H} $. Then $ \tilde{K}_j(z, \zeta) $ converges to $ \tilde{K}_{\mathcal{H}}(z, \zeta) $ uniformly on compact subsets of $ \mathcal{H} \times  \mathcal{H} $. 
\end{thm}
\begin{proof}
	Since $ D_j' $  also converges to half space in the Carath\'eodory kernel sense, by the Carath\'eodory kernel convergence theorem  \cite{OneComplexVariableIIConway}, if $ f_j: D_j \longrightarrow \D $ and $ f: \h \longrightarrow \D $ are the Riemann maps at $ z \in \h  $ for large $ j $, then $ f_j $ converges to $ f $ uniformly on compact subsets of $ \h $ as $ j \to \infty $. Hence $ f_j^{\prime}(0) $ also converges to $ f^{\prime}(0) $ as $  j \to \infty $. 
	
	We know that 
	\[
	\lambda_{D'_j}(z) = \frac{1}{f_j^{\prime}(0)} \,\, \mbox{and} \,\, \lambda_{\h}(z) = \frac{1}{f^{\prime}(0)},
	\] 
	where $ \lambda_D $ stands for the hyperbolic metric of the domain $ D $. From this we see that the hyperbolic metrics of $ D_j' $ converges to the hyperbolic metric of $ \h $ as $ j \to \infty $, consequently, for the hyperbolic metric and the span metric of a simply connected domain coincide, we obtain $ s_{D_j'}(z) $ converges to $ s_{\h}(z) $ as $ j \to \infty $. By the pull-back formula, we have
	\[
	\frac{s_{D_j'}(z)}{s_{D_j}(z)} = \frac{ s_{ U \cap D } ( T^{-1}_j(z) ) (-\psi(p_j))}{ s_{D}(T^{-1}_j(z))(-\psi(p_j)) } = \frac{ s_{ U \cap D } ( T^{-1}_j(z) )}{ s_{D}(T^{-1}_j(z)) },
	\]
	where $ T^{-1}_j(z) = -\psi(p_j)z + p_j $, this converges to the boundary point $ p $ as $ j \to \infty $.
	By Proposition~\ref{T:LocalizationOfTheSpanMetric}, we have
	\[
	\lim_{ j \to \infty }\frac{ s_{ U \cap D } ( T^{-1}_j(z) )}{ s_{D}(T^{-1}_j(z))}= 1
	\]
	Hence 
	\begin{equation}\label{E:PoitwiseConvergenceOfTheSpanMetric}
	s_{D_j}(z) \to s_{\h}(z)
	\end{equation} 
	as $ j \to \infty $, for all $ z \in \h $ pointwise.
	We shall note this convergence for later use.  
	\medskip
	
	To show the uniform convergence of $ \tilde{K}_j(\zeta, z) $ to $ \tilde{K}_{\h}(\zeta, z) $ on compact subsets of $ \h \times \h $, as a first step, we shall show that  $ \tilde{K}_j(\zeta, z) $ is locally uniformly bounded. \medskip
	
	We choose two balls $ B(z_0, r_0) $ and $ B(\zeta_0, r_0) $ of radius $ r_0 $ with centres $ z_0, \zeta_0 \in \h $ such that their closure is contained in $ \h $. Hausdorff convergence of $ D_j $ to $ \h $ guarantees that for all large $ j $, $ B(z_0, r_0) $ and $ B(\zeta_0, r_0) $ are contained in $ D_j $.\medskip
	
	By the monotonicity of the span metric
	\[
	s_{D_j}(z) \leq s_{B(z_0,r_0)}(z) = \left(\frac{r_0}{r_0^2 - |z|^2}\right)^2 \,\,\, \text{and}\,\,\, s_{D_j}(\zeta) \leq s_{B(\zeta_0,r_0)}(z) = \left(\frac{r_0}{r_0^2 - |\zeta|^2}\right)^2
	\]
	for all $ z \in B(z_0,r_0)$, $ \zeta \in  B(\zeta_0,r_0) $, and for $j$ large. 
	For $ 0 < r < r_0 $ and $ (z, \zeta)  \in \overline{B}(z_0, r) \times \overline{B}(\zeta_0, r) $, we have
	\[
	s_{D_j}(z) \leq \left(\frac{r_0}{r_0^2 - r^2}\right)^2
	\]
	and 
	\[
	s_{D_j}(\zeta) \leq \left(\frac{r_0}{r_0^2 - r^2}\right)^2
	\] 
	for $j$ large.
	Now, to bound the reduced Bergman kernels , we use the fact that $ s_{D_j}(z) = \pi \tilde{K}_j(z, z) $, whence we get
	\[
	\tilde{K}_j(z, z) \leq \frac{1}{\pi} \left(  \frac{r_0}{r_0^2 - r^2}\right)^2
	\]
	for all $ z \in \overline{B}(z_0,r) $ and 
	\[
	\tilde{K}_j(\zeta, \zeta) \leq \frac{1}{\pi} \left( \frac{r_0}{r_0^2 - r^2}\right)^2
	\]
	for all $ \zeta \in \overline{B}(\zeta_0,r) $ and for $j$ large. Now use the  Cauchy-Schwarz inequality, to obtain
	\[
	| \tilde{K}_j(\zeta, z) |^2 \leq | \tilde{K}_j(\zeta, \zeta)| | \tilde{K}_j(z, z)|,
	\]
	and this gives the local uniform bound of the reduced Bergman kernels, i.e.,
	\[
	| \tilde{K}_j(\zeta, z) | \leq \frac{1}{\pi} \left(  \frac{r_0}{r_0^2 - r^2}\right)^2
	\]
	for all  $ (z, \zeta) \in \overline{B}(z,r) \times \overline{B}(\zeta_0,r) $, and for $ j $ large. 
	Therefore, the collection of holomorphic functions $ \{\tilde{K}_j(\zeta, z) \} $ in $ \zeta, \bar z $ variables is a normal family. 
	We use the formula $ \pi \tilde{K}_j(z, z) = s_{D_j}(z) $, and note that, by (\ref{E:PoitwiseConvergenceOfTheSpanMetric}), the sequence of span metrics $ s_{D_j} $ converges to the span metric $ s_{\h} $ of the half space as $ j \to \infty $ pointwise, to conclude that 
	\[
	\tilde{K}_j(z, z) \to \tilde{K}_{\h}(z, z)
	\] 
	as $ j \to \infty $, i.e., $ \tilde{K}_j $ converges to $ \tilde{K}_{\h} $ pointwise diagonally. From this we infer that any limit function $ K(\zeta, z) $ of $ \tilde{K}_j(\zeta, z) $ is equal to $ \tilde{K}_{\h} $ diagonally, i.e.,
	\[
	K(\zeta, z) = \tilde{K}_{\h}(z, z)
	\]
	for all $ z \in \h $.
	Now taking power series expansion of $ K $ and $ \tilde{K}_{\h} $ at $ (0, 0) $ and equating them along the diagonal, we obtain that all the corresponding coefficients of these two power series are same. This gives that any limiting function of the collection $ \tilde{K}_j $ is equal to $ \tilde{K}_{\h} $ in a neighbourhood of $ (0, 0) $, and hence, by the identity theorem, on $ \h \times \h $. Therefore, we conclude that $ \tilde{K}_j $ converges to $ \tilde{K}_{\h} $ uniformly on compact subsets of $ \h \times \h $.
\end{proof}
\medskip

\begin{proof}[Proof of Theorem~\ref{T:MainThm1}]
	Note that the affine maps $ T_j^{-1} $ are biholomorphism from $ D_j $ onto $ D $, hence the pull-back metric
	\[
	(T_j^{-1})^*(s_{D})(z) = s_{D_j}(z)
	\]
	for all $ z \in D_j$ for all $ j $. By substituting $ z = 0 $, we obtain
	\[
	s_{D}(p_j)(-\psi(p_j)) = s_{D_j}(0).
	\]
	Since $ s_{D_j}(z) = \pi \tilde{K}_{D_j}(z, z) $ and $ s_{\h}(z) = \pi \tilde{K}_{\h}(z, z) $, and in Theorem~\ref{T:ConvergenceOfTheReducedBergmanKernel}, we have seen that $ \tilde{K}_j $ converges to $ \tilde{K}_{\h} $ uniformly on compact subsets of $ \h \times \h $, therefore using the pull-back, we get
	\[
	s_{D}(p_j)(-\psi(p_j)) \to s_{\h}(0)
	\]
	as $ j \to \infty $. Since $ p_j $ is an arbitrary sequence converging to $ p $, we have
	\[
	\lim_{z \to p}s_{D}(z)(-\psi(z)) = s_{\h}(0).
	\]
	Hence 
	\[
	s_D(z)  \approx 1/{\rm dist} (z, \partial D)
	\]
	for $z$ close to $p$.
\end{proof}

\begin{prop}\label{P:ConvergenceReducedBergmanKernelScaledDomains}
	Let $ \tilde{K}_j(z, \zeta) $ be the sequence of reduced Bergman kernels of the scaled domains $ D_j $ and $ \tilde{K}_{\mathcal{H}}(z, \zeta) $ be the reduced Bergman kernel of the half space $ \mathcal{H} $. Then all the partial derivatives of $ \tilde{K}_j(z, \zeta) $ converge to the corresponding partial derivatives derivatives of $ \tilde{K}_{\mathcal{H}}(z, \zeta) $ uniformly on compact subsets of $ \mathcal{H} \times \mathcal{H} $. 
\end{prop}
\begin{proof}
	Recall that the Cauchy-inequality of a holomorphic function, in particular $ \tilde{K}_j(\zeta, z) $, in two variables $ \zeta, \bar z $ is given by
	\[
	\left|\frac{\partial^{m +n}\tilde{K}_j(\zeta, z)}{\partial z^m\partial \bar{z}^n }\right| \leq
	\frac{m!n!}{4\pi^2}\sup_{(\xi_1, \xi_2) \in D^2} |\tilde{K}_j(\xi_1, \xi_2)|\frac{1}{r_1^m r_2^n},
	\]
	whenever the closure of the bidisk $ D^2 := D^2((z_0, \zeta_0); (r_1, r_2)) $ centred at $ (\zeta_0, z_0) $ of radius vector $ (r_1, r_2) $ is contained in  $D_j \times D_j$.
	Therefore, we have
	\[
	\left|\frac{\partial^{m +n}}{\partial z^m\partial \bar{z}^n }( \tilde{K}_j - \tilde{K}_{\h} )(\zeta, z)\right| 
	\leq
	\frac{m!n!}{4\pi^2}\sup_{(\xi_1, \xi_2) \in D^2} |     \tilde{K}_j(\xi_1, \xi_2) - \tilde{K}_{\h}(\xi_1, \xi_2) |\frac{1}{r_1^m r_2^n},
	\]
	since $ \tilde{K}_j - \tilde{K}_{\h} $ is a holomorphic function in two variables $ \zeta, \bar z $ in a neighbourhood of $ D^2 $ for all large $ j $. Hence the uniform convergence of $ \tilde{K}_j $ to $ \tilde{K}_{\h} $ on compact subsets of $ \h \times \h $ consequently shows that all the partial derivatives of $ \tilde{K}_j $ converges to the corresponding partial derivatives of $ \tilde{K}_{\h} $ as $ j \to \infty $.
	\medskip
\end{proof}
\begin{cor}\label{C:ConvergenceDerivativesScaledDomains}
	Let $\{D_j\}$ be the sequence of scaled domains. Then the sequence of the span metrics $s_{D_j}$ of the domains $D_j$ converges to the span metric $s_{\mathcal{H}}$ of the half-space $ \mathcal{H} $ uniformly on  compact subsets of $\mathcal{H}$. Moreover, all the partial derivatives of $s_{D_j}$ converge to the corresponding partial derivatives of $s_{\mathcal{H}}$ uniformly on  compact subsets of $\mathcal{H}$.
\end{cor}
\begin{proof}
	As $ s_{D_j}(z) = \pi \tilde{K}_j(z, z) $, the proof of this claim follows from Proposition~\ref{P:ConvergenceReducedBergmanKernelScaledDomains}.
\end{proof}

\begin{cor}\label{C:ConvergenceHigherOrderCurvatreScaledDomains}
	Let $\{ D_j\} $ be the sequence of scaled domains.
	Suppose $ \kappa_n(z: s_{D_j}) $ be the higher-order curvatures  of the span metric of $ D_j $ and $ \kappa_n(z : s_{\mathcal{H}}) $ of the half space $ \mathcal{H} $. Then 
	\[
	\kappa_n(z: s_{D_j}) \to \kappa_n(z : s_{\mathcal{H}}) = -\left(\prod_{k = 1}^{n + 1}k!\right)^2, 
	\]
	as $ j \to \infty $.
	In particular,
	\[
	\kappa(z : s_{D_j}) \to \kappa(z : s_{\mathcal{H}}) = -4,
	\]
	as $ j \to \infty $. 
\end{cor}
\begin{proof}
	As we have seen, in Corollary~\ref{C:ConvergenceDerivativesScaledDomains}, that $ s_{D_j}  \to s_{\h}$ uniformly on compact subsets of $ \h $, and hence $ 1/s^{(n + 1)(n + 2)/2}_{D_j} \to 1/s^{(n + 1)(n + 2 )/2}_{\h}  $ uniformly on compact subsets of $ \h $ as $ j \to \infty $. 
	
	Now, to show that $ \kappa_n(z: s_{D_j}) \to \kappa_n(z: s_{\h}) $, the only thing remains is to show that the determinant term in the expression of  $ \kappa_n(z: s_{D_j}) $ converges to the determinant term of in the expression of $ \kappa_n(z: s_{\h}) $ uniformly on compact subsets of $ \h $ as $ j \to \infty $, but this follows from the fact that the former determinant is sum of the products of the partial derivatives of $ s_{D_j} $ with respect to $ z $ and $ \bar z $ and Corollary~\ref{C:ConvergenceDerivativesScaledDomains}.
\end{proof}
\begin{proof}[Proof of Theorem~\ref{T:MainThm2}]
	Since higher-order curvatures are conformal invariants, the proof follows from Corollary~\ref{C:ConvergenceHigherOrderCurvatreScaledDomains}.
\end{proof}

\section{Acknowledgement}
The author would like to thank Kaushal Verma for fruitful discussions.


\end{document}